\documentclass[10pt,twoside]{article}

\usepackage{amsmath}
\usepackage{amssymb} 
\usepackage{latexsym}
\usepackage{array}
\usepackage{ifthen}
\usepackage{amsthm}
\usepackage{graphicx}
\usepackage{empheq}
\usepackage{paralist}
\usepackage{cite}
\usepackage{bigints}
\usepackage{dsfont}
\usepackage{empheq}
\usepackage{amssymb}
\usepackage{cases}
\usepackage{enumitem}
\allowdisplaybreaks

\usepackage{hyperref}

\newtheorem{defi}{\sc Definition}[section]

\newtheorem{teo}{\sc Theorem}[section]

\newtheorem{obs}{\sc Remark}[section]

\newcommand{\Ra}{\mathbb{R}}            
\newcommand{\tint}{\displaystyle\int}
\newcommand{\tsum}{\displaystyle\sum}

\makeatletter
\evensidemargin \oddsidemargin
\makeatother

\begin{document}
\thispagestyle{empty}
\setcounter{page}{1}

\noindent
{\footnotesize {\rm Submitted to\\[-1.00mm]
{\em Dynamics of Continuous, Discrete and Impulsive Systems}}\\[-1.00mm]
http:monotone.uwaterloo.ca/$\sim$journal} $~$ \\ [.3in]


\begin{center}
{\large\bf Controllability of suspension bridge model proposed by Lazer and Mckenna under the influence of impulses, delays, and non-local conditions}


\vskip.20in

Walid Zouhair$^{1}$\ and\ Hugo Leiva$^{2}$ \\[2mm]
{\footnotesize
$^{1}$Cadi Ayyad University Faculty of Sciences Semlalia LMDP, UMMISCO (IRD-UPMC), B.P. 2390 Marrakesh, Morocco, walid.zouhair.fssm@gmail.com.\\[5pt]
$^{2}$Universidad Yachay Tech, San Miguel de Urcuqui, School of Mathematical Sciences and Information Technology, Ecuador, hleiva@yachaytech.edu.ec , hleiva@ula.\\
}
\end{center}

{\footnotesize
\noindent
{\bf Abstract.}  The main purpose of this paper is to prove controllability for the model proposed by Lazer and Mckenna under the influence of impulses, delay, and nonlocal conditions. First,  we study the approximate controllability by employing a technique that pulls back the control solution to a fixed curve in a short time interval. Subsequently, based on Banach Fixed Point Theorem we investigate the exact controllability.\\[3pt]
{\bf Keywords.} Lazer and Mckenna's model, suspension bridges, controllability, impulsive semilinear evolution equation, delays, nonlocal conditions, Banach Fixed Point Theorem. \\[3pt]
{\small\bf AMS (MOS) subject classification:} Primary 34K35; Secondary 37L05.}

\vskip.2in

\section{Introduction}\label{sec1}
Impulsive dynamic systems are a type of hybrid system for which the trajectory admits discontinuities at certain instants due to sudden jumps of the state called pulses (see more in \cite{BMJHSN}). The dynamic behavior of many systems in real life can be characterized by abrupt changes that appear suddenly, such as heartbeats, drug flows, the value of stocks, impulse vaccination, and bonds on the stock market. In the past few decades, many authors worked on existence, stability, and controllability results for impulsive dynamic systems, we mention \cite{ABWZ,LHWZME,JSAVU,MMAK,vkmmadb,RASHDO,cseelgmlzw1,JSAVU2}. In particular, there are some relevant studies on impulsive systems with delay, and non-local conditions, see for instance \cite{ANMBHH,Hugo2018,xzxhzl,Lalvay}), and the references therein.

When a real-life problem is mathematically modeled, there are always intrinsic phenomena that are not taken into account, and which can affect the behavior of such a model. For example, in the case of suspension bridges, the resonance phenomenon is well known by its external forces that can have the same natural frequency as the bridge, causing it to vibrate until it breaks. To solve this problem,  several authors launched different studies on the modeling of suspension bridges and a large
number of results were established see e.g., \cite{YSCKCJPJM,JDPJM,GF2015,LAPJM} . In \cite{LAPJM} Lazer and Mckenna proposed the following second-order differential equation:
\begin{equation}
\left\{\begin{array}{l}
w^{\prime \prime}+c w^{\prime}+d w_{x x x x}+k w^{+}= p(t, x) \quad 0<x<1, t \in \mathbb{R}, \\
w(t, 0)= w(t, 1)= w_{x x}(t, 0)=w_{x x}(t,1)=0 \quad t \in \mathbb{R},
\end{array}\right.
\end{equation}
where $d >0,$ $c>0,$ $k>0$, and $ p\,:\Ra \times [0,1] \rightarrow \Ra$ is continuous, and bounded function. The nonlinear function $w^+$ denotes the function which is $w$, if $w$ is positive, and zero if $w$ is negative.

As a result of the famous Tacoma Bridge case in 1940, (see the federal report on the failure \cite{AOHTV}), researchers focused on the analytic properties of solutions of suspension bridge models, we refer our readers to \cite{SHDSJH,PJMKSM,GJACLLPJM,SHDSJH2,LH}. In \cite{LH}, H. Leiva gives sufficient conditions for the exact controllability of the following controlled suspension bridge equation under nonlinear action
\begin{equation}
\left\{\begin{array}{c}
w^{\prime \prime}+c w^{\prime}+d w_{x x x x}+k w^{+}=p(t, x)+u(t, x)+f(t, w, u(t, x)) \\
0<x<1 \\
w(t, 0)=w(t, 1)=w_{x x}(t, 0)=w_{x x}(t, 1)=0, \quad t \in \mathbb{R}.
\end{array}\right.
\end{equation}
But even so, we all know that bridges are always under the influence of abrupt changes, it would go into a large oscillation under the impulse of a single gust, also the fact that large vertical oscillation could instantly change into torsional, those abrupt changes are the main interest of our paper.  For those reasons, our main goals are to consider intrinsic phenomena such as impulses, delays, and non-local conditions in the model described above, and prove that controllability is preserved under certain conditions imposed on these new disturbances. For $t\neq t_k$ 
\begin{equation}\label{1.1}
\left\{ 
\begin{array}{lll}
\omega^{\prime \prime}+c \omega^{\prime}+d \omega_{x x x x}+k \omega^{+}=p(t, x)+u(t, x)+f(t,\tilde{\omega}_{t},\tilde{\omega}^{\prime}_{t}, u(t, x)), &  \text { in }  (0,1)_{T}, \\[2mm]
\omega(s, x)+h_{1}\left(\omega\left(\tau_{1}+s, x\right), \ldots, \omega\left(\tau_{q}+s, x\right)\right)=\rho_{1}(s, x) \\[2mm]
\omega^{\prime}(s, x)+h_{2}\left(\omega^{\prime}\left(\tau_{1}+s, x\right), \ldots, \omega^{\prime}\left(\tau_{q}+s, x\right)\right)=\rho_{2}(s, x),  &\text { in }  (0,1)_{-r} ,\\[2mm]
\omega^{\prime}(t_{k}^{+},x)=\omega^{\prime}(t_{k}^{-},x)+I_{k}(t_{k},\omega(t_k,x),\omega^{\prime}(t_k,x),u(t,x)), &k=1,\dots,m,\\[2mm]
\omega(t, 0)=w(t, 1)=\omega_{x x}(t, 0)=\omega_{x x}(t, 1)=0,   &\text { on } (0,T],
\end{array}
\right.
\end{equation}
where $(0,1)_{T} =  (0,T] \times (0,1),$ $(0,1)_{-r} = [-r,0] \times (0,1),$ $d>0,$ $c>0,$ $k>0,$ $0<t_1<t_2<\dots<t_m< T$, $0<\tau_1<\dots<\tau_q< r <T$, and the distributed control $u$ belongs to $L^{2}([0,T];\mathcal{X}),$ with $ \mathcal{X} = L^2(0,1),$  $p: (0,1)_T \longrightarrow  \Ra$ is continuous, and bounded function. The function $f: (0, T]\times\mathcal{PC}_r\times \Ra \longrightarrow  \Ra$  represents the non-linear perturbation of the differential equation in the system, and the functions $h_i: \mathcal{PC}_{r}^{q} \longrightarrow  \mathcal{PC}_{r}, \,\, i=1,2$ indicates the behavior in the non-local conditions.  The functions $(\rho_1,\rho_2)$ belongs to the spaces $\mathcal{PC}_r$ represents the historical past on the time interval $[-r,0]$. The functions  $I_{i}: [0,T]\times \Ra^{3} \longrightarrow \Ra,\, i=1\cdots, m ,$ represent the instantaneous impulses. Here, $(\tilde{\omega}_{t},\tilde{\omega}^{\prime}_{t}):[-r,0]\longrightarrow Z_{1/2}$, stands as $\tilde{\omega}_{t}(\theta)= \omega(t+\theta),\,\tilde{\omega}^{\prime}_{t}(\theta)= \omega^{\prime}(t+\theta)$, and belongs to the Banach space
\begin{align*}
	\mathcal{PC}_r =\big\{&\rho:[-r,0]\longrightarrow Z_{1/2}: \rho\;\mbox{is continuous except in a finite number of points}
\\& \theta_{ k}, k=1,2,\dots,m\; \mbox{where the side limits exist}\; \rho(\theta_{ k}^-), \; \rho(\theta_{ k}^+)=\rho(\theta_{ k})\big\},
\end{align*}
equipped with the norm
\begin{equation*}
\|\rho\|_{r}=\sup _{t \in[-r, 0]}\|\rho(t)\|_{Z_{1/2}},
\end{equation*}
where $Z_{1/2} = \mathcal{X}^{1/2} \times \mathcal{X},$ is a Hilbert space with the norm given by 
\begin{equation*}
\left\|\left(\begin{array}{c}
a \\
b
\end{array}\right)\right\|_{Z_{1/2}}=\sqrt{\|a\|_{1/2}^{2}+\|b\|^{2}}.
\end{equation*}
Next, we define the operator $A: \mathbb{D}(A)\subset \mathcal{X} \longrightarrow \mathcal{X},$ as follows
$A\phi =\phi_{xxxx}$ with domain
\begin{align*}
\mathbb{D}(A)& =\{\phi\in \mathcal{X}\,:\, \phi, \phi_{x}, \phi_{xx}, \phi_{xxx}\mbox{ are absolutely continuous},\\
& \phi_{xxxx}\in\mathcal{X};\ \phi(0)=\phi(1)=\phi_{xx}(0)=\phi_{xx}(1)=0\}.
\end{align*}
Let $\lambda_{n} =n^4 \pi^4,$ $\phi_n(x) =\sin(n\pi x)$, $P_{n} x= \left\langle x, \phi_{n}\right\rangle\phi_{n},$ is a family of complete orthogonal projections in $\mathcal{X}$. The fractional powered space  $\mathcal{X}^{1/2}$ is defined by
\begin{equation*}
\mathcal{X}^{1/2}=D\left(A^{1/2}\right)=\left\{x \in \mathcal{X}: \sum_{n=1}^{\infty} \lambda_{n}\left\|P_{n} x\right\|^{2}<\infty\right\},
\end{equation*}
equipped with the norm
\begin{equation*}
\|x\|_{1/2}=\left\|A^{1/2} x\right\|=\left\{\sum_{n=1}^{\infty} \lambda_{n}\left\|P_{n} x\right\|^{2}\right\}^{1 / 2}, \quad x \in \mathcal{X}^{1/2}.
\end{equation*}
For all $x \in D\left(A^{1/2}\right)$
\begin{equation*}
A^{1/2} x=\sum_{n=1}^{\infty} \lambda_{n}^{\frac{1}{2}}\left\langle x, \phi_{n}\right\rangle \phi_{n}=\sum_{n=1}^{\infty} \lambda_{n}^{\frac{1}{2}} P_{n} x,
\end{equation*} 
\begin{obs} \cite{Lh2}
The operator $-A$ generates an analytic semigroup $\lbrace e^{-A t}\rbrace$ given by
\begin{equation*}
e^{-A t} x=\sum_{n=1}^{\infty} e^{-\lambda_{n} t} P_{n} x.
\end{equation*}
\end{obs}
Analogously, we define the Banach space:
\begin{align*}
	\mathcal{PC}_{r}^{q}=\big\{&\varsigma : [-r,0]\longrightarrow Z_{1/2}^{q}: \varsigma \;\mbox{is continuous except in a finite number of points } \\
&\theta_{ k}, k=1,2,\dots,m\; \mbox{where the side limits exist}\; \varsigma(\theta_{ k}^-), \; \varsigma(\theta_{ k}^+)= \varsigma(\theta_{ k})\big\},
\end{align*}
equipped with the norm
\begin{equation*}
\|\varsigma\|_{r,q}=\sup _{t \in[-r, 0]}\|\varsigma(t)\|_{Z_{1/2}^{q}}=\sup _{t \in[-r, 0]}\sum_{i=1}^{q} \|\varsigma_{i}(t)\|_{Z_{1/2}}.
\end{equation*}
\section{Preliminaries and abstract formulation of the problem}
In this segment, we present the problem formulation, and some notations, fundamental definitions, propositions, and theorems that are useful to prove the main results of the paper. Using the change of variable $\omega^{\prime}= y,$ we transform the second order equation \eqref{1.1} into the following first-order system of ordinary differential equations with impulses, delays, and nonlocal conditions in the Hilbert space $Z_{1/2} = \mathcal{X}^{1/2} \times \mathcal{X} = D(A^{1/2}) \times \mathcal{X}.$

\begin{equation}\label{2.2}
\left\{\begin{array}{ll}
z^{\prime}=\mathfrak{A} z+\mathfrak{B}u+\mathfrak{F}\left(t, \tilde{z}_{t}, u\right), & t \neq t_{k} \\
z(s)+\mathfrak{G}\left(z_{\tau_{1}}, \ldots, z_{\tau_{q}}\right)(s)=\rho(s), & s \in[-r, 0] \\
z\left(t_{k}^{+}\right)=z\left(t_{k}^{-}\right)+\mathfrak{J}_{k}\left(t_{k}, z\left(t_{k}\right), u\left(t_{k}\right)\right), & k=1, \ldots, m
\end{array}\right.
\end{equation}
where
\begin{equation}\label{2.3}
z=\left[\begin{array}{l}
w \\
y
\end{array}\right], \quad \mathfrak{B}=\left[\begin{array}{c}
0 \\
I_{X}
\end{array}\right], \quad \mathfrak{A}=\left[\begin{array}{cc}
0 & I_{X} \\
-d A & -c I_{X}
\end{array}\right], \quad \rho =\left[\begin{array}{c}
\rho_1 \\
\rho_2
\end{array}\right]
\end{equation}
$\tilde{z}_{t}$ defined as a function from $[-r,0]$ to $Z_{1/2}$ by $\tilde{z}_t(s) = z(t+s), \, -r\leq s \leq 0$,\\
\begin{equation*}
\begin{aligned}
\mathfrak{J}_{\mathrm{k}}:[0, T] \times \mathrm{Z}_{1 / 2} \times \mathrm{U} & \longrightarrow \mathrm{Z}_{1 / 2} \\
(\mathrm{t}, \phi, \mathrm{u}) & \longrightarrow \left[\begin{array}{c}
0 \\
\mathrm{I}_{\mathrm{k}}(\mathrm{t}, \phi^1(\cdot), \phi^2(\cdot), \mathrm{u}(\cdot))
\end{array}\right],\\
\mathfrak{F}:\left[0, T\right] \times \mathcal{PC}_r  \times U &\longrightarrow Z_{1 / 2}\\
(\mathrm{t},\phi , u)  &\longrightarrow \left[\begin{array}{c}
0 \\
p(\mathrm{t})-k w^{+}+f(\mathrm{t},\phi^{1}(\cdot),\phi^{2}(\cdot), u(\cdot))
\end{array}\right],
\end{aligned}
\end{equation*}
with $\phi = (\phi^1,\phi^2) \in \mathrm{Z}_{1 / 2}$, and  $U= \mathcal{X} = L^2(0,1),$\\
\begin{equation*}
\begin{aligned}
\mathfrak{G}: \mathcal{PC}_{r}^{q} & \longrightarrow \mathcal{PC}_{r} \\
\mathfrak{G}\left(\phi_{1}, \ldots, \phi_{q}\right)(s) & \longrightarrow\left(\begin{array}{l}
h_{1}\left(\phi_{1}^{1}(s,\cdot), \ldots, \phi_{q}^{1}(s,\cdot)\right. \\[2mm]
h_{2}\left(\phi_{1}^{2}(s,\cdot), \ldots, \phi_{q}^{2}(s,\cdot)\right.
\end{array}\right),
\end{aligned}
\end{equation*}
with $\phi_{j} = (\phi_{j}^{1},\phi_{j}^{2})\quad j =1,\dots,q.$\\
\begin{teo}(See \cite{LH})
The operator $\mathfrak{A}$  given by \eqref{2.3}, is the infinitesimal generator of a strongly continuous group $\lbrace \mathcal{S}(t) \rbrace_{t \in \Ra}$ given by
\begin{equation}\label{eq3.8}
S(t)z=\tsum_{n=1}^{\infty}e^{A_{n}t}P_{n}z,\ \ z\in Z_{1/2},\ t\geq 0,
\end{equation}
where $\{P_{n}\}_{n\geq 0}$ is a complete family of orthogonal projections in the Hilbert space $Z_{1/2}$.
\end{teo}

Before studying the controllability of the nonlinear system \eqref{2.2} with impulses, delay , and nonlocal conditions, we start by stating the controllability of the following unperterbed linear system
\begin{equation}\label{sys3.6}
\left\{\begin{array}{l}
z^{\prime}=\mathfrak{A} z+\mathfrak{B}u(t), \quad z \in Z_{1/2},  t \in [t_0, T], 0 \leq t_0 < T \\
z(t_0)=z^{0}.
\end{array}\right.
\end{equation}
For all $z^0 \in Z_{1/2}$ , and $u\in L^2([t_0,T];U),$ the unique mild solution of \eqref{sys3.6} is given by
\begin{equation*}
z(\mathrm{t})=\mathcal{S}(\mathrm{t-t_0}) z^{0}+\int_{t_0}^{\mathrm{t}} \mathcal{S}(\mathrm{t}-\mathrm{s}) \mathfrak{B} \mathrm{u}(\mathrm{s}) \mathrm{d} \mathrm{s}, \quad \mathrm{t} \in[t_0, T],
\end{equation*}
\begin{defi} (Exact controllability).
We say that the system \eqref{sys3.6} is exactly controllable on $[t_0, T],$ if for all $z^0,z^{*} \in Z_{1/2}$ there exists a control $u\in L^2([t_0, T]; U)$ such that, the solution $z(t)$ of \eqref{sys3.6} corresponding to $u,$ verifies $z(T)=z^{*}.$
\end{defi}
We define the controllability operator as follows:
\begin{equation}\label{eq4.4}
\mathcal{G}: L^2([t_0, T]; U)\to Z_{1/2},\ \ \mathcal{G}u=\tint_{t_0}^{T}\mathcal{S}(T-s)\mathfrak{B}\mathrm{u}(s)ds.
\end{equation}
\noindent The following theorem presents the result of exact controllability for the unperturbed linear system.
\begin{teo}\label{Control}
\cite{LH} The system \eqref{sys3.6} is exactly controllable on $[t_0, T].$ Moreover, the control $u\in L^2(t_0, T; U)$ that steers an initial state $z^0$ to a final state $z^{*}$ on $[t_0, T ]$ is given by the following formula:

\begin{equation}\label{eq4.7}
\tilde{u}(t)=\mathfrak{B}^{*}S^{*}(T-t)\tsum_{j=1}^{\infty}\mathfrak{W}_{j}^{-1}(t_0)P_{j}(z^{*}-S(T-t_0)z^0),
\end{equation}
i.e.,
\begin{equation}\label{eq4.7}
\tilde{u}(t)=\mathfrak{B}^{*}S^{*}(T-t)\mathfrak{W}_{t_0}^{-1}(z^{*}-S(T-t_0)z^0), \quad t \in [t_0, T],
\end{equation}
where $\mathfrak{W}_{t_0}: Z_{1/2} \rightarrow Z_{1/2}$ is the Gramian operator given by
$$
\mathfrak{W}_{t_0}=\int_{t_0}^{T} \mathcal{S}(T-s) \mathfrak{B}\mathfrak{B}^{*} \mathcal{S}(T-s)^{*} d s=\tsum_{j=1}^{\infty}\mathfrak{W}_{j}(t_0)P_{j}.
$$
Moreover, the operator $\Gamma : Z_{1/2} \rightarrow L^2(t_0, T; \mathcal{X})$ defined by
$$
\Gamma \xi = \mathcal{G}^{*}(\mathcal{G}\mathcal{G}^{*})^{-1}=\mathfrak{B}^{*}S^{*}(T-\cdot)\mathfrak{W}_{t_0}^{-1}\xi.
$$
is a right inverse of the controllability operator $\mathcal{G}$. i.e., $\mathcal{G}\Gamma = I.$
\end{teo}
\section{Approximate controllability}

In this section, we shall prove the approximate controllability of the following nonlinear system with impulses, delays, and non-local conditions.

\begin{equation}\label{4.9}
\left\{\begin{array}{ll}
z^{\prime}=\mathfrak{A} z+\mathfrak{B}u+\mathfrak{F}\left(t,\tilde{z}_{t}, u\right), & t \neq t_{k} \\
z(s)+\mathfrak{G}\left(z_{\tau_{1}}, \ldots, z_{\tau_{q}}\right)(\mathrm{s})=\rho(s), & s \in[-r, 0] \\
z\left(t_{k}^{+}\right)=z\left(t_{k}^{-}\right)+\mathfrak{J}_{k}\left(t_{k}, z\left(t_{k}\right), u\left(t_{k}\right)\right), & \mathrm{k}=1, \ldots, m.
\end{array}\right.
\end{equation}
To this end, we shall assume that the functions $\mathfrak{F},\mathfrak{G},\mathfrak{J}_{k}$ , and $\rho$ are smooth enough, such that the above problem admits only one mild solution according with \cite{Hugo2018}, given by
\begin{eqnarray}\label{4.10}
z(t) & = & \mathcal{S}(t)\{\rho(0)-\mathfrak{G}\left(z_{\tau_{1}}, \ldots, z_{\tau_{q}}\right)(0)\}+ \int_{0}^{t}\mathcal{S}(t-s)\mathfrak{B}u(s)ds \\
&  & + \int_{0}^{t}\mathcal{S}(t-s)\mathfrak{F}\left(s, \tilde{z}_{s}, u(s)\right)ds \nonumber \\
&  & + \sum_{0 < t_k < t} \mathcal{S}(t-t_k )\mathfrak{J}_{k}\left(t_{k}, z\left(t_{k}\right), u\left(t_{k}\right)\right) \;\; t \in [0,T], \nonumber \\
z(t)& = &\rho(t)-\mathfrak{G}\left(z_{\tau_{1}}, \ldots, z_{\tau_{q}}\right)(t), \quad\quad\quad\quad\quad\,\, t \in [-r,0].\nonumber
\end{eqnarray}
\begin{defi} (Approximate controllability).
The system \eqref{4.9} is said to be approximately controllable on $[0, T],$ if for every $\epsilon >0,$ every $\rho \in \mathcal{PC}_r $ , and a final state $z^{*} \in Z_{1/2}$ there exists a control $u\in L^2([0, T]; U)$ such that, the solution $z(t)$ of \eqref{4.9} corresponding to $u,$ verifies $\|z(T)-z^{*}\| \leq \epsilon .$
\end{defi}
\begin{teo}\label{theo3}
Assume that there exist $\alpha_1 , \beta_1 \geq 0$ , and $\mathcal{H} \in \mathcal{C} \left(\mathbb{R}_{+}, \mathbb{R}_{+} \right)$ such that all $(t,\phi , u) \in \left[0, T\right] \times \mathcal{PC}_{r}([-r,0],z_{1/2}) \times U,$ the following inequality holds
\begin{equation}\label{4.13}
\|\mathfrak{F}\left(\mathrm{t}, \phi, \mathrm{u} \right)\|_{Z_{1/2}}  \leq \alpha_1\,\mathcal{H}(\| \phi(-r) \|)+\beta_1.
\end{equation}
Then, the system \eqref{4.9} is approximately controllable on $[0, T].$
\end{teo}
\begin{proof}
We start the proof by considering $u \in L^{2}(0, T ; U)$ , and corresponding mild solution  $z(t) = z(t,0,\rho,u),$ of the initial value problem \eqref{4.9}. For $0<\varsigma<\min\lbrace T-t_{m}, r\rbrace$  small enough, we define the control  $u_{\varsigma} \in L^{2}(0, T ; U)$ as follow
 \begin{equation}
 \mathrm{u}_{\varsigma}(t)=\left\{\begin{array}{ll}
u(t), & \text { if } \quad 0 \leq t \leq T-\varsigma, \\
\tilde{u} (t), & \text { if } \quad T-\varsigma<t \leq T,
\end{array}\right.
\end{equation}
where $\tilde{u}$ is the control given by Theorem \ref{Control} that steers the unpeturbed linear system \eqref{sys3.6} from the initial state $z(T-\varsigma)$ to the final state $z^{*}$ on $[T-\varsigma, T]$.
i.e.,
$$
\tilde{u}(t)=\mathfrak{B}^{*}S^{*}(T-t)\mathfrak{W}_{T-\varsigma}^{-1}(z^{*}-S(\varsigma)z(T-\varsigma)), \quad t \in [T-\varsigma, T]
$$
The corresponding solution $z^{\varsigma}=z(t,\rho,u_{\varsigma})$ of the initial value problem \eqref{4.9} at time $T$ can be written as follows:
\begin{eqnarray*}\label{adj3}
z^{\varsigma}(T) & = & \mathcal{S}(T)\{\rho(0)-\mathfrak{G}\left(z^{\varsigma}_{\tau_{1}}, \ldots, z^{\varsigma}_{\tau_{q}}\right)(0)\}+ \int_{0}^{T}\mathcal{S}(T-s)\mathfrak{B}u_{\varsigma}(s)ds \\
&  & + \int_{0}^{T}\mathcal{S}(T-s)\mathfrak{F}\left(s, z^{\varsigma}_{s}, u_{\varsigma}(s)\right)ds + \sum_{0 < t_k < T} \mathcal{S}(T-t_k )\mathfrak{J}_{k}\left(t_{k}, z^{\varsigma}\left(t_{k}\right),u_{\varsigma}\left(t_{k}\right)\right) \nonumber\\
 &  & =\mathcal{S}(\varsigma)\bigg\{ \mathcal{S}(T-\varsigma)\{\rho(0)-\mathfrak{G}\left(z^{\varsigma}_{\tau_{1}}, \ldots, z^{\varsigma}_{\tau_{q}}\right)(0)\}+ \int_{0}^{T-\varsigma}\mathcal{S}(T-s-\varsigma)\mathfrak{B}u_{\varsigma}(s)ds \\
&  & + \int_{0}^{T-\varsigma}\mathcal{S}(T-s-\varsigma)\mathfrak{F}\left(s, z^{\varsigma}_{s}, u_{\varsigma}(s)\right)ds\\
&  & + \sum_{0 < t_k < T-\varsigma} \mathcal{S}(T-t_k-\varsigma )\mathfrak{J}_{k}\left(t_k, z^{\varsigma}\left(t_k\right), u_{\varsigma}\left(t_k\right)\right) \nonumber \bigg\}\\
&  &+ \int_{T-\varsigma}^{T}\mathcal{S}(T-s)\mathfrak{B} \tilde{u}(s)ds
 + \int_{T-\varsigma}^{T}\mathcal{S}(T-s)\mathfrak{F}\left(s, z^{\varsigma}_{s}, \tilde{u}(s)\right)ds \nonumber\\
&  & = \mathcal{S}(\varsigma)z^{\varsigma}(T-\varsigma) + \int_{T-\varsigma}^{T}\mathcal{S}(T-s)\mathfrak{B}\tilde{\mathrm{u}}(s)ds
 + \int_{T-\varsigma}^{T}\mathcal{S}(T-s)\mathfrak{F}\left(s, z^{\varsigma}_{s}, \tilde{u}(s)\right)ds. \nonumber\\
\end{eqnarray*}
On the other hand, the corresponding solution $\kappa(t)=\kappa\left(t, T-\varsigma , z^{\varsigma}(T-\varsigma), \tilde{\mathrm{u}}\right)$ of the initial value problem \eqref{sys3.6} at time T is given by:
\begin{equation}
\kappa(T)=\mathcal{S}(\varsigma) z^{\varsigma}(T-\varsigma)+\int_{T-\varsigma}^{T} \mathcal{S}(T-s)\mathfrak{B}\tilde{\mathrm{u}}(s)ds.
\end{equation}
Since \eqref{sys3.6} is exactly controllable, then,
$$ \kappa(T) = z^{*}.$$
Therefore
\begin{equation*}
z^{\varsigma}(T) - z^{*}= \int_{T-\varsigma}^{T}\mathcal{S}(T-s)\mathfrak{F}\left(s, z^{\varsigma}_{s}, \tilde{u}(s)\right)ds,
\end{equation*}
by the assumption \eqref{4.13} of the theorem,  we obtain
\begin{equation*}
\begin{aligned}
\left\|z^{\varsigma}(T) - z ^{*} \right\| \leq& \int_{T-\varsigma}^{T}\|\mathcal{S}(T-s)\|\left( \alpha_1\,\mathcal{H}(\| z^{\varsigma}(s-r) \|)+\beta_1 \right) ds,
\end{aligned}
\end{equation*}
since $0\leq \varsigma \leq r$ , and $T-\varsigma \leq s \leq T$,  then  $s-r \leq T-r \leq T-\varsigma$ , hence
\begin{equation}
z^{\varsigma}_{s}(-r)=z^{\varsigma}(s-r)=z(s-r).
\end{equation}
Therefore, for $\varsigma$ small enough, we obtain that
\begin{equation*}
\begin{aligned}
\left\|z^{\varsigma}(T) - z^{*}\right\| \leq& \int_{T-\varsigma}^{T}\|\mathcal{S}(T-s)\|\left( \alpha_1\,\mathcal{H}(\| z(\mathrm{s} -r) \|)+\beta_1 \right)\, ds 
 \leq \epsilon.
\end{aligned}
\end{equation*}
Hence, the system \eqref{4.9} is approximately controllable.
\end{proof}
\section{Exact controllability}
In this section, we use the Banach contraction mapping theorem to prove the exact controllability of the system \eqref{4.9} in case the nonlinear terms $\mathfrak{F}$ and $\mathfrak{J}_{k}$ do not depend on the control function $\mathrm{u}$, and under some different conditions that we will specify later. To this end, we consider the system with nonlinear terms independent of the control variable
\begin{equation}\label{5.1}
\left\{\begin{array}{ll}
z^{\prime}=\mathfrak{A} z+\mathfrak{B}u+\mathfrak{F}\left(t, z_{t}\right), & t \neq t_{k}, \\
z(s)+\mathfrak{G}\left(z_{\tau_{1}}, \ldots, z_{\tau_{q}}\right)(s)=\rho(s), & s \in[-r, 0] ,\\
z\left(t_{k}^{+}\right)=z\left(t_{k}^{-}\right)+\mathfrak{J}_{k}\left(t_{k}, z\left(t_{k}\right)\right), & k=1, \ldots,m.
\end{array}\right.
\end{equation}
To prove the exact controllability of the previous system, we transform this problem into a fixed point problem,  and we impose some conditions in such a way that an associated operator has a fixed point. For this purpose, we assume the existence of constants $l,L_{q}, d_{k} \, > 0,$ $k=1,\dots,p,$ such that
\begin{eqnarray*}
\left\|\mathfrak{F}\left(\mathrm{t}, \psi_1 \right) -\mathfrak{F}\left(\mathrm{t}, \psi_2 \right) \right\| &\leq& l\, \|\psi_1- \psi_2\|, \quad \psi_1, \psi_2 \in \mathcal{PC}_r, \\
\|\mathfrak{G}(y)(t)- \mathfrak{G}(v)(t)\| &\leq& L_{q} \sum_{i=1}^{q}\left\|y_{i}(t)-v_{i}(t)\right\|, \quad y, v  \in \mathcal{PC}_{r}^{q},\quad\quad (\mathcal{H})\\
\left\|\mathfrak{J}_{k}(t, y)- \mathfrak{J}_{k}(t, z)\right\| &\leq& d_{k}\|y-z\|, \quad y, z \in \mathrm{Z}_{1 / 2},
\end{eqnarray*}
From Theorem\ref{theo3} , we know that system \eqref{sys3.6} is exactly controllable on $[t_0, T]$ for $0 \leq t_0 < T$, in particular for $t_0=0$, which is equivalent that  the Gramian operator
\begin{equation*}
\mathfrak{W}=\int_{0}^{T} \mathcal{S}(T-s) \mathfrak{B} \mathfrak{B}^{*} \mathcal{S}(T-s)^{*}\, d s,
\end{equation*}
is invertible, and the steering operator $\Gamma: Z_{1/2} \longrightarrow L^{2}\left(0, T ; Z_{1/2} \right),$ defined by $\Gamma \psi = \mathfrak{B}^{*} \mathcal{S}^{*}(T -\cdot) \mathfrak{W}^{-1} \psi$ is a right inverse of the controllability operator 
$\mathcal{G} : L^{2}([0, T ]; \mathcal{U}) \longrightarrow Z_{1/2}$ defined by $\displaystyle \mathcal{G}(v)=\int_{0}^{T} \mathcal{S}(T-s) \mathfrak{B} v(s)\, d s.$
Before stating the main result of this part, let us consider the following notations:\\
\begin{equation*}
M  =\sup _{s \in[0, T]}\|\mathcal{S}(s)\| , \quad\|\Gamma\|=\sup _{s \in[0, T]}\left\|\mathfrak{B}^{*} \mathcal{S}^{*}(T-s) \mathfrak{W}^{-1}\right\|,
\end{equation*}
\begin{equation*}
N = \sum_{k =1}^{m} d_k, \quad \mathcal{C} = M L q + M T l + M N,
\end{equation*}
\begin{teo}
We assume that the following inequality is satisfied
\begin{equation}\label{5.18}
M L q + M T \|\mathfrak{B}\| \|\Gamma\|\, \mathcal{C}+ M T l + M N < 1.
\end{equation}
Then, under assumptions $(\mathcal{H})$, the system \eqref{5.1} is exactly controllable on $[0,T ]$.
\end{teo}
\begin{proof}
The controllability of system \eqref{5.1} will be equivalent to the existence of a fixed point for the following operator $\kappa :  \mathcal{PC}_{m} \rightarrow \mathcal{PC}_{m},$ defined by
\begin{eqnarray*}
(\kappa y)(t) & = & \mathcal{S}(t)\{\rho(0)-\mathfrak{G}\left(y_{\tau_{1}}, \ldots, y_{\tau_{q}}\right)(0)\}+ \int_{0}^{t}\mathcal{S}(t-s)\mathfrak{B} \Gamma \mathcal{L}(y)(s) ds \\
&  & + \int_{0}^{t}\mathcal{S}(t-s)\mathfrak{F}\left(s, y_{s}\left(.\right)\right)ds + \sum_{0 < t_k < t} \mathcal{S}(t-t_k )\mathfrak{J}_{k}\left(t_k, y\left(t_k\right)\right),
\end{eqnarray*}
where the operator $\mathcal{L} : \mathcal{PC}_{m} \rightarrow Z_{1/2},$ is defined by
\begin{eqnarray*}
\mathcal{L} y & = & z^{*}- \mathcal{S}(T)\{\rho(0)-\mathfrak{G}\left(y_{\tau_{1}}, \ldots, y_{\tau_{q}}\right)(0)\} - \int_{0}^{T}\mathcal{S}(T-s)\mathfrak{F}\left(s, y_{s} \right)ds\\
&  & - \sum_{0 < t_k < T} \mathcal{S}(T-t_k )\mathfrak{J}_{k}\left(t_{k}, y\left(t_{k}\right)\right).
\end{eqnarray*}
Now, we shall prove that $\kappa$ is a contraction mapping. In fact, let $y, \omega \in \mathcal{PC}_{m}$, then
\begin{eqnarray*}
\|(\kappa\, y)(t)-(\kappa\, \omega)(t) \| &\leq & \|\mathcal{S}(t)\|\, \| \mathfrak{G}
\left(y_{\tau_{1}}, \ldots, y_{\tau_{q}}\right)(0) - \mathfrak{G}\left(\omega_{\tau_{1}}, \ldots, \omega_{\tau_{q}}\right)(0) \|  \\
&  & + \int_{0}^{t}\|\mathcal{S}(t-s)\| \|\mathfrak{B}\| \|\Gamma\| \|\mathcal{L}(y)-\mathcal{L}(\omega)\| ds\\
&   &+ \int_{0}^{t}\|\mathcal{S}(t-s)\| \|\mathfrak{F}\left(s, y_{s}\right)-\mathfrak{F}\left(s, \omega_{s}\right) \|ds \\
&  &+ \sum_{0 < t_k < t} \|\mathcal{S}(t-t_k )\| \|\mathfrak{J}_{k}\left(t_{k}, y\left(t_{k}\right)\right)-\mathfrak{J}_{k}\left(t_{k}, \omega\left(t_{k}\right)\right)\|.
\end{eqnarray*}
Then, using the above hypotheses, we get that
\begin{eqnarray*}
\|(\kappa\, y)(t)-(\kappa\, \omega)(t) \| &\leq & M L q \|y-\omega\|+ M T \|\mathfrak{B}\| \|\Gamma\| \|\mathcal{L}(y)-\mathcal{L}(\omega)\| \\
&  &+ M T l \|y-\omega\| + M N \|y-\omega\|.
\end{eqnarray*}
Furthermore, we have the following estimate

\begin{eqnarray*}
\|\mathcal{L}(y)-\mathcal{L}(\omega)\| &\leq &  M L q \|y-\omega\|+ M T l \|y-\omega\| + M N \|y-\omega\| \\
&=& \mathcal{C}\, \|y-\omega\|.
\end{eqnarray*}
Then,
\begin{eqnarray*}
\|(\kappa\, y)(t)-(\kappa\, \omega)(t) \| &\leq & \left( M L q + M T \|\mathfrak{B}\| \|\Gamma\|\, \mathcal{C}+ M T l + M N \right) \|y-\omega\|.
\end{eqnarray*}
By \eqref{5.18}, we get that $\kappa$ is a contraction mapping, and consequently, $\kappa$ has a fixed point. That is to say,
$$\kappa(z) = z .$$
Since $u = \Gamma \mathcal{L}(z)$,  we obtain
\begin{eqnarray*}
\mathcal{G} u = \mathcal{L}(z) &=& z^{*}- \mathcal{S}(t)\{\rho(0)-\mathfrak{G}\left(z_{\tau_{1}}, \ldots, z_{\tau_{q}}\right)(0)\} - \int_{0}^{T}\mathcal{S}(T-s)\mathfrak{F}\left(s, z_{s}\left(.\right)\right)ds\\
&  & - \sum_{0 < t_k < T} \mathcal{S}(T-t_k )\mathfrak{J}_{k}\left(t_{k}, z\left(t_{k}\right)\right).
\end{eqnarray*}
Then, the solution $z(t,\rho, u)$ of the system \eqref{5.1} verifies
$$
 z(T,\rho, u) = z^*.
$$
\end{proof}
\section{Conclusions and remarks}
The oscillation of suspension bridges is modeled in a second-order differential equation. This phenomenon is a real control system, it comes with many perturbations that can lead to catastrophic failures. In this work, we have taken into account the intrinsic phenomena such as impulses, delays, and non-local conditions, and we have proved that controllability is robust in the presence of all the above perturbations, under certain different conditions that we specified.

Finally, in real-life problems, the impulse starts abruptly at a certain moment of time and remains active on a finite time interval, however the time of the action is little. Such an impulse is known as a non-instantaneous impulse (NII) (see \cite{LHWZME}). It would be of much interest to investigate the impulsive controllability of the system \eqref{1.1} with non-instantaneous impulses.


\begin{thebibliography}{999}
\bibitem{AOHTV}
\newblock Amman, O. H.,  von K\'{a}rm\'{a}n, T., and  Woodruff, G. B.,
\newblock The failure of the Tacoma Narrows bridge, 1941.

\bibitem{ANMBHH}
\newblock  Abada, N.,  Benchohra, M., and  Hammouche, H.,
\newblock Existence results for semilinear differential evolution equations with impulses and delay,
\newblock\emph{ Cubo (Temuco),} \textbf{12},(2010), 1$-$17.

\bibitem{BMJHSN}
\newblock  Benchohra, M., Henderson, J., and  Ntouyas, S.,
\newblock Impulsive differential equations and inclusions,
\newblock \emph{ New York: Hindawi Publishing Corporation,} 2000.

\bibitem{ABWZ}
         \newblock  Ben Aissa,  A.,  Zouhair, W.,
         \newblock Qualitative properties for the $1-D$ impulsive wave equation: controllability and observability,
         \newblock{\em Quaest. Math.}, (2021), doi: \href{https://doi.org/10.2989/16073606.2021.1940346}{10.2989/16073606.2021.1940346}.

\bibitem {YSCKCJPJM}
\newblock  Choi, Y. S.,  Jen, K. C., and  McKenna, P. J.,
\newblock The structure of the solution set for periodic oscillations in a suspension bridge model,
\newblock \emph{IMA J. Appl. Math.} \textbf{47},(1991), 283$-$306.

\bibitem{RASHDO}
\newblock  Chorfi, S. E.,  El Guermai, G.,  Maniar, L., Zouhair, W.,
\newblock Impulsive  null approximate controllability for heat equation with dynamic boundary conditions, 
\newblock{\em Math. Control Rel. Fields}, (2022), doi: \href{https://www.aimsciences.org/article/doi/10.3934/mcrf.2022026}{10.3934/mcrf.2022026}.

\bibitem{cseelgmlzw1}
\newblock  Chorfi, S. E.,  El Guermai, G.,  Maniar, L., Zouhair, W.,
\newblock Logarithmic convexity and impulsive controllability for the one-dimensional heat equation with dynamic boundary conditions, 
\newblock{\em IMA J. Math. Control Inf.},(2022) , doi: \href{https://doi.org/10.1093/imamci/dnac013}{10.1093/imamci/dnac013}.

\bibitem {SHDSJH}
\newblock Doole, S. H., and  Hogan, S. J.,
\newblock Non-linear dynamics of the extended Lazer-McKenna bridge oscillation model,
\newblock \emph{Dynamics and Stability of Systems,} \textbf{15},(2000), 43$-$58.

\bibitem {SHDSJH2}
\newblock Doole S. H. and  Hogan,  S. J.,
\newblock A piece wise linear suspension bridge model: nonlinear dynamics and orbit continuation,
\newblock \emph{Dynamics and stability of systems,} \textbf{11},(1996), 19$-$47.

\bibitem{GF2015}
\newblock Gazzola, F.,
\newblock Mathematical models for suspension bridges,
\newblock \emph{ MSA Springer } 2015.

\bibitem{GJACLLPJM} 
\newblock  Glover, J., Lazer,  A. C., and McKenna,  P. J.,
\newblock Existence and stability of large scale nonlinear oscillations in suspension bridges,
\newblock \emph{Z. Angew. Math. Phys.} \textbf{40},(1989), 172$-$200.

 \bibitem{JSAVU}
 \newblock Jose,S. A.,  Tom, A., Ali, M. S., Abinaya, S., and  Sudsutad, W.,
\newblock Existence, uniqueness and stability results of semilinear functional special random impulsive differential equations,
\newblock\emph{ Dyn. Contin. Discrete Impuls. Syst. A: Math. Anal.} (2021), 269$-$293.

\bibitem{JSAVU2}
 \newblock Jose, S. A., and Usha, V.,
\newblock Existence and uniqueness of solutions for special random impulsive differential equation,
\newblock\emph{J. Appl. Sci. Comput,} \textbf{5}, (2018), 14$-$23.
 
\bibitem {JDPJM}
\newblock Jacover, D., and  McKenna, P. J.,
\newblock Nonlinear torsional flexings in a periodically forced suspended beam,
\newblock\emph{ J. Comput. Appl. Math.} \textbf{52},(1994), 241$-$265.

\bibitem{vkmmadb}
\newblock  Kavitha, V.,  Arjunan, M. M., and  Baleanu, D.,
\newblock Controllability of nonlocal non-autonomous neutral differential systems including non-instantaneous impulsive effects in $\Ra^{n}$,
\newblock \emph{Analele Stiint. ale Univ.} \textbf{28},  (2020), 103$-$121.


\bibitem{LAPJM}
\newblock  Lazer, A. C. and  McKenna, P. J.,
\newblock Large-amplitude periodic oscillations in suspension bridges: some new connections with nonlinear analysis,
\newblock \emph{SIAM Rev. Soc. Ind. Appl. Math.} \textbf{32},(1990), 537$-$578.

\bibitem{Lalvay}
\newblock   Lalvay, S.,   Padilla-Segarra, A.,  Zouhair, W.,
\newblock On the existence and uniqueness of solutions for non-autonomous semi-linear systems with non-instantaneous impulses, delay, and non-local conditions,
\newblock{\em Miskolc Math. Notes}, \textbf{23}, (2022), 295$–$310.	

\bibitem{LH}
\newblock  Leiva, H.,
\newblock Exact controllability of the suspension bridge model proposed by Lazer and McKenna,
\newblock \emph{J. Math. Anal. Appl.} \textbf{309},(2005), 404$-$419.

\bibitem {Hugo2018}
\newblock  Leiva, H.,
\newblock Karakostas Fixed Point Theorem and the  Existence of Solutions for Impulsive Semilinear Evolution Equations with Delays and Nonlocal Conditions,
\newblock \emph{Commun. Math. Anal.} \textbf{21},(2018), 68$–$91.

\bibitem{LHWZME} 
\newblock  Leiva, H.  Zouhair, W., Cabada, D.,
\newblock Existence, uniqueness and controllability analysis of Benjamin-Bona-Mahony equation with non instantaneous impulses, delay and non local conditions,
\newblock {\em J. Math. Control. Sci. Appl.}, \textbf{7}, (2021), 91$-$108.

\bibitem{Lh2}
\newblock  Leiva, H.,
\newblock A lemma on C0-semigroups and applications,
\newblock \emph{Quaest. Math.} \textbf{26},(2003), 247$-$265.

\bibitem {PJMKSM}
\newblock  McKenna, P. J. and  Moore, K. S.,
\newblock The global structure of periodic solutions to a suspension bridge mechanical model,
\newblock \emph{IMA J. Appl. Math.} \textbf{67},(2002), 459$-$478.

\bibitem{MMAK} 
    \newblock  Muslim, M. and  Kumar, A.,
    \newblock Existence and controllability results to second order neutral differential equation with non-instantaneous impulses.
    \newblock \emph{J. Control Decis.} \textbf{7},(2020), 286$-$308.
    
\bibitem{xzxhzl}
\newblock  Zhang, X.,  Huang, X., and  Liu, Z.,
\newblock The existence and uniqueness of mild solutions for impulsive fractional equations with nonlocal conditions and infinite delay,
\newblock \emph{Nonlinear Anal. Hybrid Syst.} \textbf{4},(2010), 775$-$781.
\end{thebibliography}
\end{document}